\documentclass[12pt]{amsart}
\usepackage{amsmath,amsfonts}
\usepackage{amscd}
\usepackage{amsthm}

\usepackage{geometry}                
\usepackage{enumerate}
\usepackage{amssymb}

\newtheorem{theorem}{Theorem}

\newtheorem{lemma}[theorem]{Lemma}
\newtheorem{corollary}[theorem]{Corollary}
\theoremstyle{remark}

\newtheorem{remarks}[theorem]{Remarks}
\theoremstyle{definition}
\newtheorem{definition}[theorem]{Definition}

\newcommand{\horbun}{H}
\newcommand{\dist}{\varrho}

\newcommand{\R}{\mathbb R}

\newcommand{\Z}{\mathbb Z}

\newcommand{\lnorm}{\left\Vert}
\newcommand{\rnorm}{\right\Vert}
\newcommand{\biglnorm}{\bigl\Vert}
\newcommand{\bigrnorm}{\bigr\Vert}

\newcommand{\lset}{\left\{}
\newcommand{\rset}{\right\}}
\newcommand{\lpar}{\left(}
\newcommand{\rpar}{\right)}
\renewcommand{\lbrack}{\left[}
\renewcommand{\rbrack}{\right]}
\newcommand{\lip}{\left<}
\newcommand{\rip}{\right>}
\newcommand{\biglpar}{\bigl(}
\newcommand{\bigrpar}{\bigr)}
\newcommand{\rest}[1]{\bigl|_{#1}}

\newcommand{\Lie}[1]{{\mathfrak{#1}}}
\newcommand{\Centre}{\mathfrak{Z}}

\newcommand{\ad}{\operatorname{ad}}

\newcommand{\Aut}{\operatorname{Aut}}

\newcommand{\Der}{\operatorname{Der}}
\newcommand{\ConfDer}{\operatorname{ConfDer}}
\newcommand{\IsoDer}{\operatorname{IsoDer}}
\newcommand{\Prol}{\operatorname{Prol}}

\newcommand{\ConfAut}{\operatorname{ConfAut}}
\newcommand{\Ad}{\operatorname{Ad}}

\newcommand{\flow}[1]{\phi_t^{#1}}
\newcommand{\dotflow}[1]{\dot\phi_t^{#1}}

\numberwithin{theorem}{section}
\numberwithin{equation}{section}

\begin{document}

\title{Conformal maps of Carnot groups}
\author{Michael G.\ Cowling}
\address{School of Mathematics and Statistics\\ University of New South Wales\\UNSW Sydney 2052\\ Australia}
\author{Alessandro Ottazzi}
\address{CIRM, Fondazione Bruno Kessler\\Via Sommarive 15\\I-38123 Trento}

\begin{abstract}
If $f$ is a conformal mapping defined on a connected open subset $\Omega$ of  a Carnot group $G$, then either $f$ is the composition of a translation, a dilation and an isometry, or $G$ is the nilpotent Iwasawa component of a real rank $1$ simple Lie group $S$, and $f$ arises from the action of $S$ on $G$, viewed as an open subset of $S/P$, where $P$ is a parabolic subgroup of $G$ and $NP$ is open and dense in $S$.
\end{abstract}

\maketitle

\section{Introduction}

The study of conformal and quasi\-conformal maps on nilpotent Lie groups began with G.~D.~Mostow's work on strong rigidity \cite{Mostow}.
P.~Pansu \cite{Pansu} developed the theory, defining Carnot groups and Carnot--Carath\'eodory spaces, and showed that in some cases, quasi\-conformal maps are automatically conformal.
Later, M.~Gromov used Carnot groups to describe the structure of groups of polynomial growth \cite{Gromov}.
In a parallel development, N. Tanaka (and his school) developed the theory of prolongation of graded Lie algebras \cite{Tanaka2} with a view to studying the equivalence of CR manifolds.
Tanaka's theory of prolongations underpins much work on parabolic geometry---see, for instance, \cite{Yamaguchi}.
In a second parallel development, E.~M.~Stein and G.~B.~Folland \cite{Folland-Stein74, Folland-Stein} (and others), developed analysis on stratified groups; in a third development, non-Riemannian geometry arose in the study of the mechanics of nonholonomic systems \cite{Bellaiche, Montgomery}.
Carnot groups are now a topic of considerable interest in their own right---see, for example, \cite{Balogh-et-al, Balogh-Tyson-Warhurst, Capogna-Cowling, Capogna-Pauls-Tyson, Danielli-Garofalo-Nhieu-Adv}.

In this paper, we consider conformal maps on Carnot groups.
Ottazzi and B.~Warhurst \cite{Ottazzi_Warhurst} have already shown that the vector space of vector fields that generate conformal local flows is finite-dimensional.
We make this more precise, and show that either the Carnot group under consideration  is the Iwasawa $N$ group of a real-rank-one noncompact simple Lie group, or all the conformal maps are affine, that is, compositions of  automorphisms and left translations.

In Section 2, we describe Carnot groups and Tanaka prolongation.
Carnot groups are characterised by an inner product on the first layer of the stratification of their Lie algebras; we show that this inner product has a canonical extension to the whole Lie algebra.
In Section 3, we characterise vector fields that generate a local flow of conformal mappings.
and show that either they form a real-rank-one simple Lie algebra, or they generate affine flows only.
This enables us to show, in Section 4, that every conformal mapping defined on a connected open subset $\Omega$ of $G$ is either the restriction to $\Omega$ of an affine map or it extends real analytically to a conformal map of $G\setminus\{p\}$, for a single point $p$.
The latter case occurs only if $G$  is one of the Iwasawa $N$ groups of a real-rank-one simple Lie group.

\section{Carnot groups and Tanaka prolongations}

We review the definitions of Carnot groups and Tanaka prolongations.
We first consider stratified Lie groups and their Lie algebras from an algebraic point of view, and then equip a stratified Lie group with the Carnot--Carath\'eodory distance to obtain a Carnot group.
We then look at some special maps of Carnot groups and their Lie algebras.

\subsection{Stratified Lie groups and algebras}
Let $G$ be a stratified Lie group of step $s$.
This means that $G$ is connected and simply connected, and its Lie algebra $\Lie{g}$ admits an $s$-step stratification:
\[
\Lie{g}= \Lie{g}_{-1}\oplus \cdots\oplus \Lie{g}_{-s},
\]
where $[\Lie{g}_{-j}, \Lie{g}_{-1}] =\Lie{g}_{-j-1}$ when $1\leq j \leq s$, while  $\Lie{g}_{-s}\neq \{0\}$ and $\Lie{g}_{-s-1}=\{0\}$; this implies that $\Lie{g}$ is nilpotent.
To avoid degenerate cases, we assume that the dimension of $G$ is at least $3$.
The identity of $G$ is written $e$, and we view the Lie algebra $\Lie{g}$ as the tangent space at the identity $e$.

For each $t \in \R^+$, the dilation $\delta_t : \Lie{g} \to \Lie{g}$ is defined by setting $\delta_t(X):=t^jX$ for every $X\in \Lie{g}_{-j}$  and every $j=1,\dots,s$, and then extending to $\Lie{g}$ by linearity.
The dilation  $\delta_t$ is an automorphism of $\Lie{g}$.
We also write $\delta_t$ for the automorphism of $G$ given by $\mathrm{exp} \circ {\delta_t} \circ \mathrm{exp}^{-1}$.

The \emph{descending central series} of $\Lie{g}$ is defined inductively, by $\Lie{g}^{(1)} = \Lie{g}$, and then
\[
\Lie{g}^{(j+1)} = [\Lie{g}, \Lie{g}^{(j)}]
\]
when $j \geq 1$.
We define subspaces $\Lie{g}^{[j+]}$ and $\Lie{g}^{[j-]}$ of $\Lie{g}$ by
\[
\Lie{g}^{[j+]} = \sum_{k \geq j} \Lie{g}_{-k}
\quad\text{and}\quad
\Lie{g}^{[j-]} = \sum_{k \leq j} \Lie{g}_{-k} .
\]
The dimension of $\Lie{g}^{[j-]}$ is written $d_j$.\
In the following lemmas, we present some well known properties of stratified Lie algebras, including the (short) proofs for the reader's convenience.
\
\begin{lemma}\label{lem:all-automorphisms}
Suppose that $T$ is a homomorphism of the stratified Lie algebra $\Lie{g}$.
Then $T(\Lie{g}^{[j+]}) \subseteq  \Lie{g}^{[j+]}$. 
\end{lemma}

\begin{proof}
It is easy to prove by induction that $\Lie{g}^{(j)} =  \Lie{g}^{[j+]}$.
Homomorphisms of Lie algebras preserve the descending central series.
\end{proof}

Throughout, we write $\Omega$ for an arbitrary nonempty connected open subset of $G$.
The differential of a differentiable map $f : \Omega \to G$ is written $f_*$.
We denote by $L_p$ the left translation by $p$ in $G$, that is, $L_p q = pq$ for all $q\in G$.
Each $X$ in $\Lie{g}$ then induces a left-invariant vector field, denoted $\tilde X$,  equal to $(L_p)_*(X)$ at each point $p \in G$.
The set $\tilde{\Lie{g}}$ of all left-invariant vector fields with vector field commutation is isomorphic to $\Lie{g}$, and it inherits the  stratification of $\Lie{g}$.
The subbundle  $\horbun  G$ of the tangent bundle $TG$, where $\horbun_p= (L_p)_*(\Lie{g}_{-1})$, is called the \emph{horizontal distribution}.
We say that a smooth map $f: \Omega \to G$ is \emph{a contact map} if $f_*$ preserves $\horbun G$.

\begin{lemma}
Suppose that $e \in \Omega$, and that $f: \Omega \to G$ is a contact map on a stratified Lie group $G$.
Then
\begin{equation}\label{eq:action-of-f_*}
(f_*)_e (\Lie{g}_{-j}) \subseteq   \Lie{g}^{[j-]}
\end{equation}
for each positive integer $j$.
\end{lemma}

\begin{proof}
We prove \eqref{eq:action-of-f_*} by induction on $j$.
Since $f$ is a contact map, $f_*$ preserves horizontal vector fields, and the result holds when $j = 1$.
Suppose now that $f_*$ sends $\tilde{\Lie{g}}_{-j}$ into $\tilde{\Lie{g}}^{[j-]}$ if $j < k$, and take $X$ in $\Lie{g}_{-1}$ and $Y$ in $\Lie{g}_{-j}$.
Then there are smooth functions $a_k$, where $1 \leq k \leq d_1$, and  $b_k$ where $1 \leq k \leq d_j$, such that
\[
 f_* \tilde X =  \sum_{k \leq d_1} a_k \tilde X_k
\quad\text{and}\quad
f_* \tilde Y =  \sum_{k' \leq d_j} b_{k'} \tilde X_{k'} .
\]
In the following formula, the indices run over the same ranges.
Clearly,
\begin{equation}\label{eq:contact-commutator}
\begin{aligned}
f_* ( [ \tilde X ,  \tilde Y ] )
&= [ f_* \tilde X ,  f_* \tilde Y ] \\
&=  \lbrack \sum a_k \tilde X_k , \sum b_{k'} \tilde X_{k'} \rbrack  \\
&= \sum a_k (\tilde X_k b_{k'}) \tilde X_{k'}
	- \sum b_{k'} (\tilde X_{k'}  a_{k}) \tilde X_{k} \\
&\qquad +  \sum a_k b_{k'}  [\tilde X_k , \tilde X_{k'} ] .
\end{aligned}
\end{equation}
Thus $f_*$ sends $\tilde{\Lie{g}}_{-j-1}$ into $\tilde{\Lie{g}}^{[(j+1)-]}$, and \eqref{eq:action-of-f_*} follows.
\end{proof}

The proof of the following complementary result is immediate.

\begin{lemma}\label{lem:good-automorphisms}
 Suppose that $T$ is a homomorphism of a stratified Lie  group $G$.
 The following are equivalent:
\begin{enumerate}[(i)]
\item
$T$ is a contact map;
\item
$dT({\Lie{g}}_{-k})  \subseteq \Lie{g}_{-k}$, for each positive integer $k$;
\item
$T$ commutes with dilations.
\end{enumerate}
\end{lemma}

\subsection{Carnot groups}
We fix a scalar product $\lip\cdot,\cdot\rip$ on $\Lie{g}_{-1}$, and define a left-invariant scalar product on each horizontal space $\horbun_p$ by setting
\begin{align}
\lip V, W \rip_p = \lip (L_{p^{-1}})_*(V),(L_{p^{-1}})_*(W) \rip  \label{scalarprod}
\end{align}
for all $V, W \in \horbun_p$.
The left-invariant scalar product gives rise to a left-invariant \emph{sub-Riemannian} or  \emph{Carnot--Carath\'eodory} distance function $\dist$ on $G$.
To define this, we first say that a smooth curve $\gamma$ is  \emph{horizontal} if $\dot\gamma(t)\in \horbun_{\gamma(t)}$ for every $t$.
Then we define the distance $\dist(p,q)$ between points $p$ and $q$ by
\[
\dist(p, q) := \inf\int_0^1 \biglpar \lip  \dot\gamma(t), \dot\gamma(t) \rip_{\gamma(t)} \bigrpar^{1/2}  \,dt ,
\]
where the infimum is taken over all horizontal curves $\gamma: [0, 1] \to G$ such that $\gamma(0) = p$ and $\gamma(1) = q$.
The stratified group $G$, equipped with the distance $\dist$, is known as a \emph{Carnot group}; we usually omit mention of $\dist$.

We now equip the whole of $\Lie{g}$ with an inner product.
Let $\lset E_1, \dots, E_{d_1}\rset$ be an orthonormal basis for $\Lie{g}_{-1}$.
For each positive integer $j$, we equip  $\Lie{g}_{-1}\otimes^j$ with the standard inner product, for which an orthonormal basis is formed of the tensors $E_{i_1} \otimes \dots \otimes E_{i_j}$, where each index varies over $\lset 1, \dots, {d_1}\rset$.
Define the linear projection $P_j:\Lie{g}_{-1}\otimes^j\to \Lie{g}_{-j}$ by the requirement that
\begin{equation}\label{def:projection-P_j}
P_j(X_1\otimes X_2\otimes \cdots \otimes X_j)=[\dots [X_1,X_2] \dots, X_j]
\end{equation}
for all $X_i\in \Lie{g}_{-1}$.
Then $\Lie{g}_{-j}$ is isomorphic to $\ker P_j^\perp$, and we can therefore provide each $\Lie{g}_{-j}$ with the restriction to $\ker P_j^\perp$ of the inner product.
We extend the inner product to $\Lie{g}$ by keeping the different spaces $\Lie{g}_{-j}$ orthogonal; the new inner product is still denoted by $\lip\cdot,\cdot\rip$.
The next lemma follows from the definitions.

\begin{lemma}\label{lem:tensor-product}
The projection $P_j:\Lie{g}_{-1}\otimes^j\to \Lie{g}_{-j}$ and the inner products just defined have the following properties:
\begin{enumerate}[(i)]
\item
 for all $\tau \in \Lie{g}_{-1}\otimes^j$, we have $\lnorm P_j(\tau) \rnorm \leq  \lnorm \tau \rnorm$.

\item
for all $W \in \Lie{g}_{-j}$, there exists $\tau \in \Lie{g}_{-1}\otimes^j$ such that $P_j(\tau) = W$ and $\lnorm \tau \rnorm \leq \lnorm W \rnorm$.
\end{enumerate}
\end{lemma}

\begin{definition}
We define a Riemannian metric on $G$ by the formula
\begin{align}
\lip V, W \rip_p = \lip (L_{p^{-1}})_*(V),(L_{p^{-1}})_*(W) \rip  \label{scalarprod}
\end{align}
for all $V, W \in T_p$.
\end{definition}

\subsection{Morphisms and affine maps}
We now discuss homomorphisms and automorphisms  of Carnot groups.
Since $G$ is simply connected and the exponential map $\exp: \Lie{g} \to G$ is a diffeomorphism, $T$ is a homomorphism or automorphism of $G$ if and only if its differential $dT$ is a homomorphism or automorphism of $\Lie{g}$, and $T = {\mathrm{exp}} \circ  {dT} \circ {\mathrm{exp}}^{-1}$.

For a homomorphism of $\Lie{g}$, preserving all the subspaces $\Lie{g}_{-j}$ of the stratification is equivalent to commuting with dilations.
We use the adjective ``strata-preserving'' to describe such homomorphisms, both at the algebra level and at the group level.
We denote by $\Aut(\Lie{g})$ the Lie group of strata-preserving automorphisms of $\Lie{g}$, and by $\Aut(G)$ the corresponding Lie group of automorphisms of $G$.

The Lie algebra of the group $\Aut(\Lie{g})$ is the Lie algebra of strata-preserving derivations of $\Lie{g}$, which we denote by $\Der(\Lie{g})$.
The set $\lset \delta_t : t \in  \R^+\rset$ of dilations is a one parameter subgroup of $\Aut(\Lie{g})$, whose Lie algebra is generated by  the derivation $H \in \Der (\Lie{g})$ defined by $H(X):=jX$, for every $X\in \Lie{g}_{-j}$ and  $j=1,\dots,s$.
In particular, $\delta_{e^s}=\exp(s H) $ for every $s \in \R$.

\begin{lemma}\label{lem:stratified-norms}
Suppose that $T$ is a strata-preserving automorphism of $\Lie{g}$.
Then
\[
\biglnorm T \rest{\Lie{g}_{-j}} \bigrnorm \leq \biglnorm T \rest{\Lie{g}_{-1}} \bigrnorm ^j .
\]
\end{lemma}

\begin{proof}
The standard extension of $T$ to a linear mapping on $\Lie{g}_{-1}\otimes^j$ has operator norm at most $\biglnorm T \rest{\Lie{g}_{-1}} \bigrnorm ^j$.
The result now follows from Lemma \ref{lem:tensor-product}.
\end{proof}

We now consider various special kinds of automorphisms of $G$.
For the moment, we say that an automorphism $T$ of $G$ is distance-preserving if
\[
\dist(Tx,Ty) = \dist(x,y)
\quad\forall x,y \in G .
\]
Since the distance on $G$ is left-invariant, and is defined in terms of horizontal curves, $T$ is distance-preserving if and only if its differential $dT$ on $\Lie{g}$ is isometric on $\Lie{g}_{-1}$, that is,
\[
\lnorm dT(X) \rnorm = \lnorm X \rnorm
\quad\forall X \in \Lie{g}_{-1} .
\]

\begin{lemma}\label{lem:isometry-higher-strata}
If $T$ is an automorphism of $\Lie{g}$ that is isometric  on $\Lie{g}_{-1}$, then $T$ is isometric on $\Lie{g}$, that is,
\begin{equation*}
 \lnorm T(X) \rnorm = \lnorm  X \rnorm
 \quad\forall X \in \Lie{g}.
\end{equation*}
\end{lemma}

\begin{proof}
This follows from Lemma \ref{lem:stratified-norms} applied to $T$ and $T^{-1}$.
\end{proof}

We write $\IsoDer(\Lie{g})$ for the space of derivations of $\Lie{g}$ that are skew-symmetric on $\Lie{g}_{-1}$.
From Lemma \ref{lem:isometry-higher-strata}, every element of $\IsoDer(\Lie{g})$ is actually skew-symmetric on all $\Lie{g}$.

Now we consider conformal automorphisms of $\Lie{g}$.
An automorphism $T$ of $\Lie{g}$ is conformal with dilation factor $t$ if and only if
\[
\lnorm T(X) \rnorm = t \lnorm X \rnorm
\quad\forall X \in \Lie{g}_{-1} .
\]
The Lie algebra of the group $\ConfAut(\Lie{g})$ of conformal automorphisms of $\Lie{g}$ is the Lie algebra $\ConfDer(\Lie{g})$ of conformal derivations,  which is equal to $\R H + \IsoDer(\Lie{g})$.

If $T$ is an automorphism of $G$, then $dT$ is conformal, with dilation factor $t$, if and only if
\[
\dist(Tx,Ty) = t \dist(x,y)
\quad\forall x,y \in G ,
\]
or, equivalently, $\delta_t ^{-1}T$ is distance-preserving.

\begin{corollary}
If $T$ is a conformal automorphism of $\Lie{g}$, then
\begin{equation*}
 \lnorm T(X) \rnorm = \lnorm  \delta_{t} X \rnorm
 \quad\forall X \in \Lie{g},
\end{equation*}
where $T$ dilates by $t$ on $\Lie{g}_{-1}$.
\end{corollary}

\begin{proof}
This follows immediately from Lemma \ref{lem:isometry-higher-strata}.
\end{proof}

\begin{corollary}\label{cor:conformality-higher-strata}
If $T$ is a conformal automorphism of $\Lie{g}$, and $\lnorm T(X) \rnorm = \lnorm X \rnorm$ for one nonzero element $X$ of $\Lie{g}$, then
\begin{equation*}
 \lnorm T(X) \rnorm = \lnorm  X \rnorm
 \quad\forall X \in \Lie{g}.
\end{equation*}
\end{corollary}

\begin{proof}
This follows immediately from the previous corollary.
\end{proof}

An affine map is a map composed of an automorphism and a translation.
To some extent, it is irrelevant whether the translation is on the left or the right, as left and right translations differ by a conjugation.
However, since left translations are distance-preserving, while right translations are not (unless $G$ is abelian), it is more convenient to deal with left translations.

\subsection{General maps}
Recall that $\Omega$ denotes a connected open subset of $G$.
We recall that a continuous map $f:\Omega \to G$ is \emph{Pansu differentiable} at $p \in \Omega$ if the limit
\[
 {D}f(p)(q)=\lim_{t\rightarrow 0^+} \delta_t^{-1}\circ L_{f(p)}^{-1}\circ f \circ L_{p}\circ \delta_t(q)
\]
exists, uniformly for $q$ in compact subsets of $G$;  if ${D}f(p)$ exists, then it is a strata-preserving homomorphism of $G$.
Then there is a Lie algebra homomorphism  $df(p):\Lie{g} \to \Lie{g}$ such that ${D}f(p) \circ \mathrm{exp} = \mathrm{exp} \circ {d}f(p)$.
We call  ${D}f(p)$ the Pansu derivative and $df(p)$ the Pansu differential of $f$ at $p$.
By construction, both ${D}f(p)$ and ${d}f(p)$ commute with dilations, and so in particular, ${d}f(p)$ is a strata-preserving Lie algebra homomorphism.

We say that a $C^1$ map $f : \Omega \to G$ is isometric or conformal if ${d}f(p)$ is isometric or conformal for each point $p \in \Omega$; the dilation factor in conformality may vary from point to point.

\begin{lemma}
 Suppose that $f: \Omega \to G$ is $C^1$.
 Then $f$ is distance-preserving if and only if $df(p)$ is an isometry at each point $p$ in $\Omega$.
\end{lemma}

\begin{proof}
It is clear that if $f$ is isometric, then $f$ preserves the lengths of admissible curves and hence distances.
Conversely, if $f$ preserves distances, then differentiation shows that the Pansu differential is an isometry.
\end{proof}

At this point, it is clear that we may call distance-preserving maps isometries without risk of confusion.
The following result was first proved by E.~Le~Donne and Ottazzi, in greater generality \cite{LeDonne-Ottazzi}.

\begin{theorem}\label{thm:isometries}
Suppose that $f: \Omega \to G$ is an isometry.
Then $f$ is the restriction to $\Omega$ of an isometric affine mapping, that is, the composition of an isometric automorphism and a left translation.
\end{theorem}

\subsection{Tanaka prolongation}
In \cite{Tanaka2}, Tanaka introduced the prolongation of a stratified Lie algebra relative to a subalgebra $\Lie{g}_0$ of $\Der(\Lie{g})$.
The prolongation  $\Prol(\Lie{g},\Lie{g}_0)$ has the following properties:
\begin{enumerate}[(P1)]
\item     $\Prol(\Lie{g},\Lie{g}_0)=\sum_{i = -s}^t \Lie{g}_i$ is a graded Lie algebra and $\Lie{g} = \sum_{i = -s}^{-1}\Lie{g}_i$;
\item     if $U\in\Lie{g}_k$ where $k\geq 0$ and $[U,\Lie{g}_{-1}]=0$, then $U=0$;
\item     $\Prol(\Lie{g},\Lie{g}_0)$ is maximal among the Lie algebras satisfying (P1) and (P2).
\end{enumerate}
In (P1), the upper limit $t$ may be a natural number or $+\infty$.

We only consider prolongations satisfying the additional condition that $H \in \Lie{g}_0$, where $H$ is the element in $\Der(\Lie{g})$ such that $[H, X ]  =  jX$ for all $X \in \Lie{g}_{-j}$ and all $j \in \Z^+$.
The next lemma will clarify the graded structure of $\Prol(\Lie{g}, \Lie{g}_0)$.

\begin{lemma}
The operator $\ad(H)$ is diagonalisable, and $\ad(H) |_{\Lie{g}_{k}}  = -k I$ for all $k \in \Z$.
In particular, the element $H$  is in the centre of $\Lie{g}_0$.
Hence $[\Lie{g}_j,\Lie{g}_k] \subseteq \Lie{g}_{j+k}$ for all integers $j$ and $k$, and $[\Lie{g}_0,\Lie{g}_k]=\Lie{g}_k$ for all nonzero integers $k$.
Further, the centre $\Centre(\Prol(\Lie{g}, \Lie{g}_0))$ of $\Prol(\Lie{g}, \Lie{g}_0)$ is trivial, so $\ad$ is faithful on $\Prol(\Lie{g}, \Lie{g}_0)$.
Finally, if $\Prol(\Lie{g}, \Lie{g}_0)$ is finite-dimensional, then any ideal of $\Prol(\Lie{g}, \Lie{g}_0)$ is also graded.
\end{lemma}

\begin{proof}
We have already seen that
\[
\ad(H) |_{\Lie{g}_{k}}  = -k I
\]
for all negative integers $k$, and we prove it by induction for all integers $k$, starting at $-1$.
Suppose that $k \geq 0$ and that $\ad(H) |_{\Lie{g}_{j}}  = -j I$ if $j \leq  k-1$.
We must show that $\ad(H) |_{\Lie{g}_k} = -k I$.

Take $Y \in \Lie{g}_k$ and $X \in \Lie{g}_{-1}$.
Since $[Y,X] \in \Lie{g}_{k-1}$,
\[
[[H,Y] , X]=[H,[Y, X]]+[[H,X],Y] = -(k-1) [Y,X]  + [X,Y]= -k [Y,X] .
\]
Since $\Lie{g}_k$ is completely determined by its action on $\Lie{g}_{-1}$, it follows that $\ad(H) = -k I$ on $\Lie{g}_k$.

It follows immediately that $\Centre(\Prol(\Lie{g}, \Lie{g}_0)) \subseteq \Lie{g}_0$ and $[\Lie{g}_0,\Lie{g}_k]=\Lie{g}_k$ for all nonzero integers $k$.
Further, the Jacobi identity implies that $[\Lie{g}_j,\Lie{g}_k] \subseteq \Lie{g}_{j+k}$ for all integers $j$ and $k$.
If $Z \in \Centre(\Prol(\Lie{g}, \Lie{g}_0))$, then $Z \in \Lie{g}_0$, and hence $Z$ is a derivation of $\Lie{g}$ that is null, and so $Z = 0$.

 Suppose that $\Lie{t}$ is an ideal in the finite-dimensional vector space $\Prol(\Lie{g}, \Lie{g}_0)$.
Then $\Lie{t}$ is closed in $\Prol(\Lie{g}, \Lie{g}_0)$.
If $X \in \Lie{t}$, then we may write $X = \sum_j X_j$, where $X_j \in \Lie{g}_j$.
Write $\bar\jmath$ for $\max\lset j : X_j \neq 0 \rset$.
Since $\Ad(\exp(tH)) X \in \Lie{t}$, we see that $e^{-t\bar\jmath} \sum_j e^{tj} X_j \in \Lie{t}$ and hence, taking limits, $X_{\bar\jmath} \in \Lie{t}$.
Induction shows that all components of $X$ are in $\Lie{t}$.
\end{proof}

\subsection{A graded version of the Levi--Malcev theorem}
We now suppose that the Lie algebra $\Prol(\Lie{g}, \Lie{g}_0)$ is finite-dimensional.
We say that a subalgebra $\Lie{a}$ of $\Prol(\Lie{g}, \Lie{g}_0)$ is $H$-graded if $\Lie{a} = \sum_{j} (\Lie{a} \cap \Lie{g}_{j})$.
We say that $X \in  \Prol(\Lie{g}, \Lie{g}_0)$ is semisimple or nilpotent if $\ad(X)$ is semisimple or nilpotent.
Recall that if $D \in \Der(\Lie{g})$, then we may write $D = D^s + D^n$, where $D^s$ is semisimple while $D^n$ is nilpotent, and both $D^s$ and $D^n$ lie in $\Der(\Lie{g})$ (see \cite[Section 1, Proposition 4]{Bourbaki-GAL7}.
We say that a Lie subalgebra $\Lie{g}_0$ of $\Der(\Lie{g})$ is  \emph{splittable} if the semisimple and the nilpotent parts of each $D \in \Lie{g}_0$ also lie in $\Lie{g}_0$.
Evidently $\Der(\Lie{g})$ is splittable, as is $\ConfDer(\Lie{g})$, because this contains semisimple elements only.
The next result is due to C.~Medori and M.~Nacinovich \cite{Medori-Naci}.

\begin{theorem}\label{thm:graded-Levi-splitting}
Suppose that $\Lie{g}_0$ is a splittable Lie subalgebra of $\Der(\Lie{g})$, and that $ \Prol(\Lie{g}, \Lie{g}_0)$ is finite-dimensional.
Then the solvable radical $\Lie{r}$ of $\Prol(\Lie{g}, \Lie{g}_0)$ is $H$-graded, and it is possible to choose an $H$-graded semisimple subalgebra $\Lie{s}$ such that $\Prol(\Lie{g}, \Lie{g}_0) = \Lie{s} + \Lie{r}$.
\end{theorem}

\section{Conformal vector fields}
Henceforth, we suppose that $\Lie{g}_0= \ConfDer(\Lie{g})$, and denote by $\Lie{p}$  the Lie algebra $\Prol(\Lie{g} , \Lie{g}_0)$ that defines the space of conformal vector fields on $G$.
We decompose $\Lie{p}$ as $\sum_k \Lie{g}_k$.
According to Ottazzi and Warhurst \cite{Ottazzi_Warhurst}, $\Lie{p}$ is finite-dimensional (we agreed that the dimension of a Carnot group is at least $3$), and $\Lie{p}$ and the space of conformal vector fields, that is, vector fields whose local flow consists of conformal mappings, are isomorphic Lie algebras.
In this case, there is a graded isomorphism between the prolongation algebra and a Lie algebra of vector fields on $G$.
We will use upper case letters from $U$ to $Z$ to indicate elements of $\Lie{p}$ while $\breve U$, $\breve V$, and so on will indicate the corresponding vector fields: for a smooth function $u$ on $G$,
\[
(\breve U)u(p) = \frac{d}{dt} u(\flow{U} (p))\rest{t=0} \,,
\]
where $\flow{U}$ is the (local) flow associated to $U$.
More precisely, given a relatively compact open subset $\Omega$ of $G$, there is a subinterval $I$ of $\R$, which contains $0$, such that $\flow{U}(p)$ is defined in $G$ for all $p$ in $\Omega$ and all $t$ in $I$, and $\dotflow{U}(p) = (\breve U)_{q}$, where the dot denotes differentiation with respect to $t$ and $q = \flow{U}(p)$.
In particular, if $U \in \Lie{g}$, then
\[
(\breve U)u(p) = \frac{d}{dt} u(\exp(-tU)p)\rest{t=0}
\]
for all smooth functions $u$ on $G$.

\begin{lemma}\label{lem:conf-Lie-algebra-structure}
In the conformal prolongation algebra $\Lie{p}$, suppose that $X \in \Centre(\Lie{g})\setminus\{0\}$ and $U \in \sum_{k \geq 1} \Lie{g}_{k}\setminus\{0\}$.
Then $[U,X] \neq 0$.
\end{lemma}

\begin{proof}
Suppose, with a view to a contradiction, that $[U,X]=0$.

The flow generated by $\breve X$ is left translation, which is defined in all $G$ and for all time.
We denote by $\flow{U}$ the local flow corresponding to $\breve U$; we may assume that $\flow{U}(\exp(sX)p)$ is defined for all $p$ in an open set $\Omega$ and all $s$ and $t$ in a time interval $I$ that contains $0$.
Because $X$ and $U$ commute, so do the flows, and thus
\[
\flow{U} (\exp(sX)p)= \exp(sX) \flow{U}(p)
\]
for all $p \in \Omega$ and all $s$ and $t$ in $I$.
Since $X$ is in the centre of $\Lie{g}$, $\exp(sX)$ is in the centre of $G$, and so
\[
\flow{U} ( p \exp(sX))=  \flow{U}(p) \exp(sX)
\]
Taking the derivative with respect to $s$ at $0$, we deduce that
\[
d(\flow{U})  X = X,
\]
and in particular,
\begin{equation}\label{isometry}
\| d\flow{U}(p)  X\| = \|  X \| .
\end{equation}
Since $\flow{U}$ is conformal, Corollary \ref{cor:conformality-higher-strata} shows that
\[
\| d\flow{U}(p) Z \| = \|  Z \|,
\]
for every  $Z\in \Lie{g}$ and $p \in \Omega$.
This implies that $\flow{U}$ is isometric in $\Omega$, whence by Theorem \ref{thm:isometries}, the vector field $\breve U$ must be generated  by vectors in $\Lie{p} \cap \sum_{k\leq0} \Lie{g}_k$, which is the desired contradiction.
\end{proof}

\begin{theorem}\label{theorem-algebra}
Suppose that $G$ is a Carnot group.
Then either $\Lie{p} = \Lie{g} + \Lie{g}_0$ or $\Lie{p}$ is a noncompact simple Lie algebra of real rank $1$.
\end{theorem}

\begin{proof}
Suppose that $\Lie{g}_{1} = \{0\}$; then $\Lie{p} = \Lie{g} + \Lie{g}_0 = \Lie{g} + \ConfDer(\Lie{g})$.
Suppose that $\Lie{g}_{1} \neq \{0\}$; we will show that $\Lie{p}$ is simple of real rank $1$.

First, we split $\Lie{p}$ into a semisimple and a solvable part.
By Theorem \ref{thm:graded-Levi-splitting}, we may write $\Lie{p} =\Lie{s} + \Lie{r}$, where $\Lie{r}$ is the radical of $\Lie{p}$ and $\Lie{s}$ is semisimple; both $\Lie{r}$ and $\Lie{s}$ are $H$-graded.
As usual, we write $\Lie{r} = \sum_j \Lie{r}_j$.

Next, we show that $\Lie{r}_0 \subseteq \IsoDer(\Lie{g})$ and $H \in \Lie{s}$.
If $D \in \Lie{g}_0$, then $D = s H  + M$, where $M \in \IsoDer(\Lie{g})$ and $s \in \R$, and so $D$ is semisimple.
We claim that $D \notin \Lie{r}$ unless $s = 0$.
Indeed, the eigenvalues of $\ad(D)$ on $\Lie{g}_{j}$ are all of the form $x + i y$, where $x =- sj$ and $y$ is real; hence $\ad(D)$ is invertible on $\Lie{g}_{j}$ when $j \neq 0$ and $s \neq 0$.
If it were true that $D \in \Lie{r}$, then it would follow that $\Lie{g}_j \subset \Lie{r}$ whenever $j \neq 0$, since $\Lie{r}$ is an ideal.
Since $\Lie{g}_{1} \neq \{0\}$, we can find $X \in \Lie{g}_1$ and $Y \in \Lie{g}_{-1}$ such that $[X,Y] \neq 0$.
Then $[X,Y]$ would be both semisimple (since $[X,Y] \in \Lie{g}_0$) and nilpotent (since $[X,Y] \in [\Lie{r}, \Lie{r]}$); this is absurd, and so $D \notin \Lie{r}$.
We deduce that $\Lie{r}_0 \subseteq \IsoDer(\Lie{g})$ and $H \in \Lie{s}$.

Third, we write $\Lie{s}$ as $\Lie{s}_0 \oplus \Lie{s}_1\oplus \dots \oplus  \Lie{s}_\ell$, where each $\Lie{s}_j$ is simple; we will show that exactly one of the $\Lie{s}_j$ is noncompact.
Every semisimple element $S$ of a semisimple Lie algebra $\Lie{s}$ lies in a Cartan subalgebra; 
indeed, by \cite[Section 2, Proposition 10]{Bourbaki-GAL7}, a Cartan subalgebra of the commutant of $S$, which must contain $S$, is a Cartan subalgebra of $\Lie{s}$.
We take a Cartan subalgebra $\Lie{h}$  of $\Lie{s}$ that contains $H$.
Then $\Lie{h} \subseteq \Lie{g}_0$ since $[H, X ] = 0$ for all $X \in \Lie{h}$.
We may write
\[
\Lie{h} = \lpar\Lie{h} \cap \Lie{s}_0\rpar \oplus \lpar\Lie{h} \cap \Lie{s}_1\rpar \oplus \dots \oplus \lpar\Lie{h} \cap \Lie{s}_\ell\rpar,
\]
where each $\Lie{h} \cap \Lie{s}_j$ is a Cartan subalgebra of $\Lie{s}_j$.
Consider the Killing form of $\Lie{s}$; this must be negative definite on $\IsoDer(\Lie{g})$ and positive definite on $\R H$.
If there were several noncompact summands $\Lie{s}_j$, then there would be a subspace of $\Lie{h}$ on which the Killing form is positive definite of dimension at least $2$.
Thus there must be one noncompact summand, $\Lie{s}_0$ say, which contains $H$, and all the other summands are compact.
Further, $\Lie{s}_0$ is of real rank one since the centraliser of $H$ in $\Lie{s}_0$ is $\R H + \IsoDer(\Lie{g}) \cap \Lie{s}_0$, and $\IsoDer(\Lie{g}) \cap \Lie{s}_0$ is compact.

Now, we show that $\Lie{r}$ is trivial, arguing by contradiction.
Supposing otherwise, $\Lie{g}$ contains a nontrivial abelian ideal $\Lie{a}$.
Since $\Lie{a}$ is an ideal, $\ad(\Lie{g}_{-1})^k \Lie{a} \subseteq \Lie{a}$.
Because $\Lie{g}$ is nilpotent, there must be some positive integer $k$ such that $\ad(\Lie{g}_{-1})^{k-1} \Lie{a} \neq \{0\}$ while $\ad(\Lie{g}_{-1})^{k} \Lie{a} = \{0\}$.
Thus $\Lie{a} \cap \Centre(\Lie{g}) \neq\{0\}$.
Under the adjoint action of $\Lie{s}$, we can split $\Lie{a}$ into $\Lie{s}$-invariant subspaces, on each of which $\Lie{s}$ acts by an irreducible finite-dimensional representation.
By taking a suitable subspace $\Lie{v}$, we may suppose that $ \Lie{v} \cap \Centre(\Lie{g}) \neq \{0\}$.

The subspace $\Lie{v}$ is abelian because $\Lie{a}$ is, and decomposes into weight spaces for the Cartan subalgebra $\Lie{h}$; the set of weights is closed under multiplication by $-1$, by representation theory.
Since $H \in \Lie{h}$ and $\ad(H)$, acting on $\Lie{p}$, is diagonalizable with integer eigenvalues, we may write
\[
\Lie{v} = \sum_{j = -k}^k \Lie{v}_{j},
\]
where $\Lie{v}_j = \Lie{v} \cap \Lie{g}_j$.
We may suppose that $\Lie{v}_{-k} \cap \Centre(\Lie{g}) \neq  \{0\}$.
By Lemma \ref{lem:conf-Lie-algebra-structure}, $[\Lie{v}_{-k},\Lie{v}_{k}] \neq \{0\}$, which contradicts the fact that $\Lie{v}$ is abelian.
Thus $\Lie{r}$ must be trivial, and $\Lie{p} = \Lie{s}$.

Finally, we show that $\Lie{s}$ is simple.
On the one hand, $\Lie{s}_j \subseteq \Lie{g}_0$  when $j \neq 0$.
On the other hand, $\Lie{g}_{-1} \subseteq \Lie{s}$, and so $\Lie{g}_{-1} \subseteq \Lie{s}_0$.
The elements of $\Lie{s}_j$, where $j > 0$, commute with those of $\Lie{s}_0$, and so $[\Lie{s}_j, \Lie{g}_{-1}] = \{0\}$.
This contradicts the definition of $\Lie{g}_0$ as a set of derivations of $\Lie{g}$, unless $\Lie{s}_j = \{0\}$.
\end{proof}

\section{Conformal mappings on Carnot groups}
In this section, we state and prove our characterization of conformal mappings of Carnot groups, which is a consequence of Theorem \ref{theorem-algebra}.
Before we do this, we recall two facts that we will need.

First, if $S$ is a simple Lie group of real rank $1$, then $S$ has subgroups $N$, $\bar N$, $M$ and $A$ such that $N M A \bar N$ is an open dense submanifold of $S$, whose complement is of the form $wMA\bar N$ for a particular element $w$ of $S$.
Thus we may identify the quotient space $S/MA\bar N$ with the one-point compactification $N \cup \{\infty\}$, and $S$ acts on this space on the left.
The nilpotent group $N$ is a Carnot group, and the action of $S$ on $N \cup \{\infty\}$ is conformal.

Second, the nilradical is the maximal nilpotent ideal of a Lie algebra.
Every  Lie algebra homomorphism sends the nilradical into the nilradical, so every Lie algebra automorphism preserves the nilradical.

\begin{theorem}
Let $G$ be a Carnot group, let $\Omega$ be a connected open  subset of $G$, and let $f:\Omega\rightarrow G$ be a conformal mapping.
Then there are two possibilities:
\begin{enumerate}[(i)]
\item
$G$ is not the Iwasawa $N$ group of a real-rank-one simple Lie group $S$; in this case, $f$ is the restriction to $\Omega$ of a conformal affine map of $G$, and extends analytically to a conformal map on $G$.
\item
$G$ is the Iwasawa $N$ group of a real-rank-one simple Lie group $S$; in this case, $f$ is the restriction to $\Omega$ of the action of an element of $S$ on $N \cup\{\infty\}$; thus $f$ extends analytically to a conformal map on $G$ or $G \setminus\{p\}$ for some point $p$.
\end{enumerate}
\end{theorem}
\begin{proof}
Let us first observe that a polynomial vector field is determined by its restriction to any nonempty open subset of $G$; consequently, we do not need to distinguish between such vector fields on $G$ and their restrictions to open subsets.

Suppose that $f: \Omega \to G$ is conformal.
By composing with left translations, we may suppose that $e \in \Omega$ and $f(e) = e$.

Each conformal vector field $\breve U$ on $G$ induces a local flow $\flow{U}$ on a neighbourhood of $e$ in $\Omega$; we define a new conformal local flow $\psi_t$ on a neighbourhood of $e$ by conjugation
\begin{equation}\label{eq:conjuation-by-f}
\psi_t = f \circ \flow{U} \circ f^{-1}.
\end{equation}
Differentiation with respect to $t$ yields a new conformal vector field $\breve V$ on a neighbourhood of $e$ such that $\psi_t = \flow{V}$.
Further, if $\flow{U}$ fixes $e$ for all $t$, then so does $\flow{V}$.
The map $\breve f: U \mapsto V$ is an automorphism of the prolongation algebra $\Lie{p}$, and leaves invariant the subalgebra $\Lie{p}_e$ of $\Lie{p}$ corresponding to vector fields that vanish at $e$.
From \eqref{eq:conjuation-by-f}, we see that
\begin{equation}\label{eq:breve-f}
\flow{\breve f(U)} (e) =  f( \flow{U} (e)).
\end{equation}

Assume now that $G$ is not the Iwasawa $N$ group of a real-rank-one simple Lie group $S$.
By Theorem \ref{theorem-algebra}, the Lie algebra of conformal vector fields is isomorphic to $\Lie{g} + \Lie{g}_0$.
The nilradical of $\Lie{g} + \Lie{g}_0$ is $\Lie{g}$, and so $\breve f$ preserves $\Lie{g}$; the subalgebra $\Lie{p}_e$ is equal to $\Lie{g}_0$, so $\breve f$ also preserves $\Lie{g}_0$.
Clearly there is an automorphism $T$ of $G$, not necessarily strata-preserving, such that $dT = \breve f$.

Now $\flow{W}(e) = \exp(-t W)$ for all $W \in \Lie{g}$.
From \eqref{eq:breve-f}, it follows that
\[
 f(\exp(-t U)) = \exp(-t \breve f (U)) = T \exp(-t U)
\]
for all $U$ in $\Lie{g}$ and all $t \in  \R$, that is, $f$ is an automorphism.
Since $f$ is also a contact map, the automorphism is strata-preserving, by Lemma \ref{lem:good-automorphisms}, and since $f$ is conformal, the automorphism is in $\ConfAut(\Lie{g})$.

The case where $G$ is the Iwasawa $N$ group of a real-rank-one simple Lie group $S$ is essentially contained in \cite[Section 5]{CDKR}; in that paper, the assumption of real rank at least two serves only to ensure that the space of vector fields appropriate to the mappings under consideration is finite-dimensional.
\end{proof}

\begin{remarks}
The result immediately generalize to 1-quasiconformal maps, thanks to \cite{Capogna-Cowling}.
The hypothesis that $\Omega$ is connected is essential, in the sense that if $f_1: \Omega_1 \to G$ and   $f_2: \Omega_2 \to G$ are conformal, and $\Omega_1$ and $\Omega_2$ are disjoint, then $f$ is conformal; however, an analytic extension to $G$ is not possible.
The hypothesis that $G$ is a Carnot group is also important, since on compact Lie groups with a bi-invariant Riemannian metric, inversion is conformal but not affine.
If $G$ is a stratified group with a sub-Finsler metric, then a similar result holds; much as argued in \cite{LeDonne-Ottazzi}, a sub-Finsler conformal map is also a sub-Riemannian conformal map for an appropriate metric.
\end{remarks}

\def\cprime{$'$} \def\cprime{$'$}
\providecommand{\bysame}{\leavevmode\hbox to3em{\hrulefill}\thinspace}
\renewcommand{\MR}[1]{\relax{}}

\providecommand{\MRhref}[2]{%
  \href{http://www.ams.org/mathscinet-getitem?mr=#1}{#2}
}
\providecommand{\href}[2]{#2}

\end{document}